\documentclass[12pt]{article}
\usepackage{amsmath,amsfonts,amssymb,amsthm}
\newtheorem{theo}{Theorem}
\newtheorem{prop}{Proposition}

\newtheorem{coro}{Corollary}
\newcommand\la{\rho}
\numberwithin{equation}{section}

\title{A new generalization of 
binomial coefficients}
\author{Michel Lassalle\\
\small Centre National de la Recherche Scientifique\\[-0.8ex]
\small Institut Gaspard-Monge, Universit\'e de Marne-la-Vall\'ee\\[-0.8ex]
\small 77454 Marne-la-Vall\'ee Cedex, France\\[-0.8ex]
\small \texttt{lassalle@univ-mlv.fr}\\[-0.8ex]
\small \texttt{http://igm.univ-mlv.fr/{\textasciitilde}lassalle}
}
\date{}
\begin{document}
\maketitle

\begin{abstract}
Let $t$ be a fixed parameter and $x$ some indeterminate. We give some properties of the generalized binomial coefficients $\genfrac{\langle}{\rangle}{0pt}{}{x}{k}$ inductively defined by $k/x \,\genfrac{\langle}{\rangle}{0pt}{}{x}{k}=
t\genfrac{\langle}{\rangle}{0pt}{}{x-1}{k-1}
+(1-t)\genfrac{\langle}{\rangle}{0pt}{}{x-2}{k-2}$.
\end{abstract}

\section{Definition}

There are many generalizations of binomial coefficients, the most elementary of which are the Gaussian polynomials. In this note, we shall present another one-parameter extension, presumably new, encountered in the study of the symmetric groups~\cite{La}. This application will be described at the end.

Let $t$ be a fixed parameter and $x$ some indeterminate. For any positive integer $k$, we define a function $\genfrac{\langle}{\rangle}{0pt}{}{x}{k}$ inductively by $\genfrac{\langle}{\rangle}{0pt}{}{x}{k}=0$ for $k<0$, $\genfrac{\langle}{\rangle}{0pt}{}{x}{0}=1$ and
\[\frac{k}{x}\genfrac{\langle}{\rangle}{0pt}{}{x}{k}=
t\genfrac{\langle}{\rangle}{0pt}{}{x-1}{k-1}
+(1-t)\genfrac{\langle}{\rangle}{0pt}{}{x-2}{k-2}.\]
Then $\genfrac{\langle}{\rangle}{0pt}{}{x}{k}$ is a polynomial with degree $k$ in $x$ and in $t$. First values are given by
\begin{align*}
&\genfrac{\langle}{\rangle}{0pt}{}{x}{1}=tx,\qquad\quad
\genfrac{\langle}{\rangle}{0pt}{}{x}{2}=t^2\,\binom{x}{2}+(1-t)\frac{x}{2},\\
&\genfrac{\langle}{\rangle}{0pt}{}{x}{3}=t^3\,\binom{x}{3}+t(1-t)\,\binom{x}{2}-\frac{1}{3}t(1-t)x,\\
&\genfrac{\langle}{\rangle}{0pt}{}{x}{4}=t^4\,\binom{x}{4}+\frac{3}{2}t^2(1-t)\,\binom{x}{3}-\frac{1}{12}(1-t)(8t^2+3t-3)\,\binom{x}{2}+\frac{1}{8}(1-t^2)(2t-1)x.
\end{align*}
For $k$ odd, it is obvious that $\genfrac{\langle}{\rangle}{0pt}{}{x}{k}$ is divisible by $t$. 

We have easily
\[\genfrac{\langle}{\rangle}{0pt}{}{x}{k}=t^k \binom{x}{k}+(1-t)xP(x,t),\]
with $P$ a polynomial of degree $k-2$ in $x$ and $t$. For $t=1$ we recover the classical binomial product
\[\binom{x}{k}=\frac{1}{k!}\prod_{i=1}^{k} (x-i+1),\]
and when $x$ is some positive integer $n$, the binomial coefficient $\binom{n}{k}$. 

In this paper we shall present some notable properties of the generalized binomial coefficients $\genfrac{\langle}{\rangle}{0pt}{}{x}{k}$, including a generating function, a Chu-Vandermonde identity and an explicit formula. The referee has suggested that it would be interesting to obtain a $q$-analogue of our results, using $q$-shifted factorials instead of ordinary ones, in the same way than~\cite{La2} has been generalized by~\cite{Z}.

\section{Generating function}

We consider the series
\begin{align*}
G(u)&=1+tu+(1-t)\sum_{k\ge0}\frac{(-u)^{k+2}}{(k+2)!}\prod_{i=0}^k (k-i+1+(t-1)i),\\
H(u)&=1+\sum_{k\ge1}\frac{(-u)^k}{k!}\prod_{i=1}^k (k-i+1+(t-1)i).
\end{align*}
We have
\[\frac{d}{du}G(u)=t+(1-t)u\,H(u).\]

\begin{theo}
The series $G(u)$ and $H(u)$ are mutually inverse.
\end{theo}
\begin{proof}
Krattenthaler has pointed out that the statement is a consequence of Rothe identity~\cite{Chu, Guo}
\[\sum_{k=0}^n\frac{A}{A+Bk}\binom{A+Bk}{k}\binom{C-Bk}{n-k}=\binom{A+C}{n}.\]
Actually if we denote
\[X_k=-\frac{k+1}{t-2},\qquad Y_k=X_{k-2}=-\frac{k-1}{t-2},\]
we have
\begin{align*}
G(u)&=1+tu+\frac{1-t}{2-t}\sum_{k\ge 2} u^k(t-2)^k\frac{1}{Y_k+1}\binom{Y_k+1}{k},\\
&=u+\frac{1-t}{2-t}\sum_{k\ge 0} u^k(t-2)^k\frac{1}{Y_k+1}\binom{Y_k+1}{k},\\
H(u)&=\sum_{k\ge 0} u^k(t-2)^k\binom{X_k-1}{k}.
\end{align*}
Rothe identity, written for 
\[A=1+\frac{1}{t-2},\qquad B=-\frac{1}{t-2},\qquad C=-1-\frac{n+1}{t-2},\]
yields
\[\frac{1-t}{2-t}\sum_{k=0}^n\frac{1}{Y_k+1}\binom{Y_k+1}{k}\binom{X_{n-k}-1}{n-k}=\binom{-n/(t-2)}{n}=\binom{X_{n-1}}{n},\]
since $X_{n-1}=-n/(t-2)$. Therefore for $n\neq 0$ the coefficient of $u^n$ in $H(u)G(u)$ is
\begin{multline*}
(t-2)^{n-1}\binom{X_{n-1}-1}{n-1}+(t-2)^{n}\binom{X_{n-1}}{n}=\\
(t-2)^{n-1}\binom{X_{n-1}-1}{n-1}\Big(1+(t-2)\frac{X_{n-1}}{n}\Big)=0.
\end{multline*}
\end{proof}

\begin{coro}
The series $G(u)$ is the unique solution of $G(0)=1$ and 
\begin{equation*}
\frac{d}{du}G(u)=t+(1-t)\frac{u}{G(u)}.
\end{equation*}
\end{coro}

\begin{theo}
The generating function of the numbers $\genfrac{\langle}{\rangle}{0pt}{}{x}{k}$ is given by
\[\sum_{k\ge0}\genfrac{\langle}{\rangle}{0pt}{}{x}{k}u^k=(G(u))^x.\]
\end{theo} 
\begin{proof}
If we write
\[\mathcal{F}(u;x)=\sum_{k\ge0}\genfrac{\langle}{\rangle}{0pt}{}{x}{k}u^k,\]
the definition of $\genfrac{\langle}{\rangle}{0pt}{}{x}{k}$ yields
\[\frac{1}{x}\frac{d}{du}\mathcal{F}(u;x)=t\mathcal{F}(u;x-1)+(1-t)u\mathcal{F}(u;x-2).\]
Inspired by the $t=1$ case which is the classical binomial formula $\sum_{k\ge0}\binom{x}{k}u^k=(1+u)^x$, we may look for a generating function of the form $\mathcal{F}(u;x)=(F(u))^x$. Then we have
\[\frac{d}{du}F(u)=t+(1-t)\frac{u}{F(u)}.\]
We apply Corollary 1.
\end{proof}

\begin{coro}
For $k\ge 1$ we have 
\begin{equation*}
\genfrac{\langle}{\rangle}{0pt}{}{-1}{k}=\frac{(-1)^k}{k!}\prod_{i=1}^k (k-i+1+(t-1)i).
\end{equation*}
\end{coro} 
 \begin{proof}
Consequence of $\mathcal{F}(u;-1)=H(u)$.
\end{proof}

\begin{coro}
We have the generalized Chu-Vandermonde formula
\begin{equation*}
\sum_{i=0}^k\genfrac{\langle}{\rangle}{0pt}{}{x}{i}\genfrac{\langle}{\rangle}{0pt}{}{y}{k-i}=\genfrac{\langle}{\rangle}{0pt}{}{x+y}{k}.
\end{equation*}
\end{coro} 
\begin{proof}
Standard consequence of $\mathcal{F}(u;x+y)=\mathcal{F}(u;x)\mathcal{F}(u;y)$.
\end{proof}

\begin{coro}
A variant of the generalized Chu-Vandermonde formula is given by
\begin{equation*}
\sum_{i=0}^k i \, \genfrac{\langle}{\rangle}{0pt}{}{x}{i}\genfrac{\langle}{\rangle}{0pt}{}{y}{k-i}=\frac{kx}{x+y} \genfrac{\langle}{\rangle}{0pt}{}{x+y}{k}.
\end{equation*}
\end{coro} 
\begin{proof}
By definition the left-hand side is given by
\[
x \sum_{i=0}^k  \genfrac{\langle}{\rangle}{0pt}{}{y}{k-i}\Big(t\genfrac{\langle}{\rangle}{0pt}{}{x-1}{i-1}+(1-t)\genfrac{\langle}{\rangle}{0pt}{}{x-2}{i-2}\Big).\]
By the generalized Chu-Vandermonde formula, it can be written as
\[
xt\genfrac{\langle}{\rangle}{0pt}{}{x+y-1}{k-1}+x(1-t)\genfrac{\langle}{\rangle}{0pt}{}{x+y-2}{k-2}=x\frac{k}{x+y}\genfrac{\langle}{\rangle}{0pt}{}{x+y}{k}.
\]
\end{proof}

\noindent\textit{Remark. }When $m\neq 0,1$ we do not know any such simple expression for
\begin{equation*}
\sum_{i=0}^k \binom{i}{m} \, \genfrac{\langle}{\rangle}{0pt}{}{x}{i}\genfrac{\langle}{\rangle}{0pt}{}{y}{k-i}.
\end{equation*}

\begin{coro}
For $n\ge 2$ we have
\[\sum_{i=0}^n i \genfrac{\langle}{\rangle}{0pt}{}{n}{n-i}\genfrac{\langle}{\rangle}{0pt}{}{-1}{i-1}=0.\]
\end{coro} 
\begin{proof}
Since we have
\[\sum_{i=0}^n i \genfrac{\langle}{\rangle}{0pt}{}{n}{n-i}\genfrac{\langle}{\rangle}{0pt}{}{-1}{i-1}=\sum_{i=0}^{n-1} (i+1) \genfrac{\langle}{\rangle}{0pt}{}{n}{n-i-1}\genfrac{\langle}{\rangle}{0pt}{}{-1}{i},\]
the property follows from Corollaries 3 and 4, written with $x=-1$, $y=n$, $k=n-1$. Actually in that case $kx/(x+y)+1=0$.
\end{proof}

Looking for the contributions to $u^m$ in $(G(u))^{x-1}\, G(u)$ we have the following generalization of Pascal's recurrence formula
\[
\genfrac{\langle}{\rangle}{0pt}{}{x}{m}-
\genfrac{\langle}{\rangle}{0pt}{}{x-1}{m}=t\genfrac{\langle}{\rangle}{0pt}{}{x-1}{m-1}+(1-t)\sum_{k=0}^{m-2}\frac{(-1)^{k}}{(k+2)!}\genfrac{\langle}{\rangle}{0pt}{}{x-1}{m-k-2}\prod_{i=0}^k (k-i+1+(t-1)i).
\]
Similarly with $(G(u))^{x+1}\, H(u)$ we get
\[\genfrac{\langle}{\rangle}{0pt}{}{x}{m}-
\genfrac{\langle}{\rangle}{0pt}{}{x+1}{m}=\sum_{k=1}^m\frac{(-1)^k}{k!}\genfrac{\langle}{\rangle}{0pt}{}{x+1}{m-k}\prod_{i=1}^k (k-i+1+(t-1)i).
\]

\section{New properties}

Let $n$ be a positive integer. The previous properties of $\genfrac{\langle}{\rangle}{0pt}{}{n}{k}$ are very similar to those of the classical binomial coefficient $\binom{n}{k}$. However some big differences must be emphasized.

Firstly $\genfrac{\langle}{\rangle}{0pt}{}{n}{k}$ and $\genfrac{\langle}{\rangle}{0pt}{}{n}{n-k}$ are not equal. In particular $\genfrac{\langle}{\rangle}{0pt}{}{n}{n} \neq 1$. By definition we have
\[\genfrac{\langle}{\rangle}{0pt}{}{n}{n}=
t\genfrac{\langle}{\rangle}{0pt}{}{n-1}{n-1}
+(1-t)\genfrac{\langle}{\rangle}{0pt}{}{n-2}{n-2}.\]
which implies by induction
\[\genfrac{\langle}{\rangle}{0pt}{}{n}{n}=1+(t-1)\genfrac{\langle}{\rangle}{0pt}{}{n-1}{n-1},\]
and
\[\genfrac{\langle}{\rangle}{0pt}{}{n}{n}=\frac{1-(t-1)^{n+1}}{2-t}.\]

Secondly $\genfrac{\langle}{\rangle}{0pt}{}{n}{k}$ is not zero for $k>n$, but divisible by $(1-t)$. Starting from the definition we have for $k\ge 2$,
\begin{align*}
k\genfrac{\langle}{\rangle}{0pt}{}{1}{k}&=(1-t)\genfrac{\langle}{\rangle}{0pt}{}{-1}{k-2}\\
&=(1-t)\frac{(-1)^k}{(k-2)!}\prod_{i=1}^{k-2} (k-i-1+(t-1)i).
\end{align*}
By induction we get
\begin{align*}
\binom{k}{2}\genfrac{\langle}{\rangle}{0pt}{}{2}{k}&=t(1-t)\genfrac{\langle}{\rangle}{0pt}{}{-1}{k-3}, \quad k\ge 3,\\
\binom{k}{3}\genfrac{\langle}{\rangle}{0pt}{}{3}{k}&=(1-t)\big(t^2+\frac{k-1}{2}(1-t)\big)\genfrac{\langle}{\rangle}{0pt}{}{-1}{k-4}, \quad k\ge 4,\\
\binom{k}{4}\genfrac{\langle}{\rangle}{0pt}{}{4}{k}&=t(1-t)\big(t^2+\frac{5k-8}{6}(1-t)\big)\genfrac{\langle}{\rangle}{0pt}{}{-1}{k-5}, \quad k\ge 5.
\end{align*}

More generally for $k>n$ the definition yields
\begin{equation*}
\binom{k}{n}\genfrac{\langle}{\rangle}{0pt}{}{n}{k}=(1-t)f_{n,k}\genfrac{\langle}{\rangle}{0pt}{}{-1}{k-n-1},
\end{equation*}
where $f_{n,k}$ is a monic polynomial in $t$, inductively defined by $f_{1,k}=1$, $f_{2,k}=t$ and
\[f_{n,k}=tf_{n-1,k-1}+(1-t)\frac{k-1}{n-1}f_{n-2,k-2}.\]

The coefficients $\genfrac{\langle}{\rangle}{0pt}{}{n}{k}$ with $k>n$ may be written in terms of coefficients $\genfrac{\langle}{\rangle}{0pt}{}{n}{k}$ with $k\le n$. The simplest case is given below.
\begin{prop}
We have
\begin{equation*}
\frac{1}{1-t} \genfrac{\langle}{\rangle}{0pt}{}{n}{n+1}=\sum_{i=0}^n \frac{i}{i+1} \genfrac{\langle}{\rangle}{0pt}{}{n}{n-i}\genfrac{\langle}{\rangle}{0pt}{}{-1}{i-1}.
\end{equation*}
\end{prop}
\begin{proof}
Denoting the right-hand side by $h_n$, we must prove that $h_n=f_{n,n+1}/(n+1)$. Equivalently that $h_n$ is inductively defined by 
\[\frac{n+1}{n} h_n=th_{n-1}+(1-t)h_{n-2}.\]
In other words, that we have
\[\sum_{i=0}^n \frac{i}{i+1}\left(\frac{n+1}{n}\frac{n-i}{n-i}\genfrac{\langle}{\rangle}{0pt}{}{n}{n-i}-t\genfrac{\langle}{\rangle}{0pt}{}{n-1}{n-i-1}-(1-t)\genfrac{\langle}{\rangle}{0pt}{}{n-2}{n-i-2}\right)\genfrac{\langle}{\rangle}{0pt}{}{-1}{i-1}=0.\]
By the definition this may be rewritten as
\[\sum_{i=0}^n \frac{i}{i+1}\Big(\frac{n+1}{n-i}-1\Big)\frac{n-i}{n}\genfrac{\langle}{\rangle}{0pt}{}{n}{n-i}\genfrac{\langle}{\rangle}{0pt}{}{-1}{i-1}=0.\]
We apply Corollary 5.
\end{proof}

\section{Binomial expansion}

In this section we consider the expansion
\[\genfrac{\langle}{\rangle}{0pt}{}{x}{k}=t^k\binom{x}{k}+\sum_{i=1}^{k-1} c_i(k)\binom{x}{i}.\]
We give two methods for the evaluation of the coefficients $c_i(k)$, $1 \le i\le k-1$.

\subsection{First method}

The definition of $\genfrac{\langle}{\rangle}{0pt}{}{x}{k}$ may be written as
\[\sum_{i=1}^k \frac{k}{i}c_i(k) \binom{x-1}{i-1}=
t\sum_{i=1}^{k-1} c_i(k-1) \binom{x-1}{i}
+(1-t)\sum_{i=1}^{k-2} c_i(k-2)\binom{x-2}{i}.\]
Using the classical identity
\[\binom{x}{i}=\sum_{m=0}^i (-1)^m \binom{x+1}{i-m},\]
and identifying the coefficients of $\binom{x-1}{i-1}$ on both sides, we obtain 
\[\frac{k}{i}\,c_i(k)=tc_{i-1}(k-1)+(1-t)\sum_{m=0}^{k-i-1} (-1)^m\
c_{i+m-1}(k-2).\]
This relation may be used to get $c_i(k)$ inductively.

For $k\ge6$ this recurrence yields
\begin{align*}
\genfrac{\langle}{\rangle}{0pt}{}{x}{k}
&=t^k\, \binom{x}{k}-\frac{k-1}{2}t^{k-2}(t-1)\,\binom{x}{k-1}\\
&+\frac{k-2}{3}t^{k-4}(t-1)\Big(t^2+\frac{3(k-3)}{8}(t-1)\Big)\,\binom{x}{k-2}\\
&-\frac{k-3}{4}t^{k-6}(t-1)\Big(t^4+\frac{4k-13}{6}t^2(t-1)+\frac{1}{6}\binom{k-4}{2}(t-1)^2\Big)\,\binom{x}{k-3}\\
&+\frac{k-4}{5}t^{k-8}(t-1)\Big(t^6+\frac{65k-229}{72}t^4(t-1)\\
&\hspace{2 cm}+\frac{5(2k-9)}{48}(k-5)t^2(t-1)^2+\frac{5}{64}\binom{k-5}{3}(t-1)^3\Big)\,\binom{x}{k-4}\\
&-\frac{k-5}{6}t^{k-10}(t-1)\Big(t^8+\frac{66k-251}{60}t^6(t-1)
+\frac{85k^2-853k+2148}{240}t^4(t-1)^2\\
&\hspace{2 cm}+\frac{4k-23}{48}\binom{k-6}{2}t^2(t-1)^3+\frac{3}{80}\binom{k-6}{4}(t-1)^4\Big)\,\binom{x}{k-5}\\
&+\ldots.
\end{align*}
It appears empirically that $c_{k-i}(k)$, $i\neq 0$, has the form
\[c_{k-i}(k)=(-1)^{i}\frac{k-i}{i+1} t^{k-2i}(t-1)\sum_{m=0}^{i-1} a_m (k,i)\,t^{2(i-m-1)}(t-1)^m,\]
with $a_m (k,i)$ a polynomial in $k$ with degree $m$. 

\subsection{Second method}

A much better expression of $c_i(k)$, $1 \le i\le k-1$, may be obtained by using the relation
\begin{equation}
\frac{1}{1-t}\binom{k}{n}\genfrac{\langle}{\rangle}{0pt}{}{n}{k}=f_{n,k}\genfrac{\langle}{\rangle}{0pt}{}{-1}{k-n-1},\qquad k\ge n+1.
\end{equation}
For instance $c_1(k)$ is given by 
\[\genfrac{\langle}{\rangle}{0pt}{}{1}{k}=c_1(k),\]
hence
\[\frac{1}{1-t}k\,c_1(k)=\frac{1}{1-t}k\genfrac{\langle}{\rangle}{0pt}{}{1}{k}
=\genfrac{\langle}{\rangle}{0pt}{}{-1}{k-2}.\]
Similarly we have 
\[\genfrac{\langle}{\rangle}{0pt}{}{2}{k}=c_2(k)+2c_1(k),\]
which yields
\[\frac{1}{1-t}\binom{k}{2}c_2(k)=t\genfrac{\langle}{\rangle}{0pt}{}{-1}{k-3}-(k-1)\genfrac{\langle}{\rangle}{0pt}{}{-1}{k-2}.\]

This property is generalized by the following explicit formula.
\begin{theo}
Let $f_{n,k}$ be the polynomial in $t$ and $k$ inductively defined by $f_{1,k}=1$, $f_{2,k}=t$ and
\[f_{n,k}=tf_{n-1,k-1}+(1-t)\frac{k-1}{n-1}f_{n-2,k-2}.\]
We have
\[\genfrac{\langle}{\rangle}{0pt}{}{x}{k}=t^k\binom{x}{k}+\sum_{i=1}^{k-1} c_i(k)\binom{x}{i},\]
with $c_i(k)$ given by
\[\frac{1}{1-t}\binom{k}{i}c_i(k)=\sum_{m=1}^{i} (-1)^{i-m}f_{m,k} \binom{k-m}{i-m}\genfrac{\langle}{\rangle}{0pt}{}{-1}{k-m-1}.\]
\end{theo}
\begin{proof}
From
\[\genfrac{\langle}{\rangle}{0pt}{}{i}{k}=\sum_{m=1}^i c_m(k)
\binom{i}{m}, \quad i\le k-1,\]
we deduce by inversion
\begin{equation*}
c_i(k)=\sum_{m=1}^i 
(-1)^{i-m} \binom{i}{m}\genfrac{\langle}{\rangle}{0pt}{}{m}{k},\quad i\le k-1,
\end{equation*}
which is a direct consequence of the classical identity
\[\sum_{m=p}^i(-1)^{i-m}\binom{i}{m}\binom{m}{p}=\delta_{ip}.\]
Therefore we have
\begin{align*}
\frac{1}{1-t}\binom{k}{i}c_i(k)&=\sum_{m=1}^i 
(-1)^{i-m} \frac{1}{1-t} \binom{k}{i}\binom{i}{m}\genfrac{\langle}{\rangle}{0pt}{}{m}{k}\\
&=\sum_{m=1}^i 
(-1)^{i-m} \binom{k-m}{i-m}\frac{1}{1-t} \binom{k}{m}\genfrac{\langle}{\rangle}{0pt}{}{m}{k}.
\end{align*}
We apply (4.1).
\end{proof}

The first values are given by
\begin{align*}
\frac{1}{1-t}k\,c_1(k)&=\genfrac{\langle}{\rangle}{0pt}{}{-1}{k-2},\\
\frac{1}{1-t}\binom{k}{2}c_2(k)&=t\genfrac{\langle}{\rangle}{0pt}{}{-1}{k-3}-(k-1)\genfrac{\langle}{\rangle}{0pt}{}{-1}{k-2},\\
\frac{1}{1-t}\binom{k}{3}c_3(k)&=\big(t^2+\frac{k-1}{2}(1-t)\big)\genfrac{\langle}{\rangle}{0pt}{}{-1}{k-4}-t(k-2)\genfrac{\langle}{\rangle}{0pt}{}{-1}{k-3}+\binom{k-1}{2}\genfrac{\langle}{\rangle}{0pt}{}{-1}{k-2},\\
\frac{1}{1-t}\binom{k}{4}c_4(k)&=t\big(t^2+\frac{5k-8}{6}(1-t)\big)\genfrac{\langle}{\rangle}{0pt}{}{-1}{k-5}-(k-3)\big(t^2+\frac{k-1}{2}(1-t)\big)\genfrac{\langle}{\rangle}{0pt}{}{-1}{k-4}\\
&+t\binom{k-2}{2}\genfrac{\langle}{\rangle}{0pt}{}{-1}{k-3}-\binom{k-1}{3}\genfrac{\langle}{\rangle}{0pt}{}{-1}{k-2}.
\end{align*}

\section{An application}

A partition $\la= (\la_1,...,\la_r)$
is a finite weakly decreasing
sequence of nonnegative integers, called parts. The number
$l(\la)$ of positive parts is called the length of
$\la$, and $|\la| = \sum_{i = 1}^{r} \la_i$
the weight of $\la$.

For any partition $\la$ and any integer $1 \le i \le l(\la)+1$, we denote by $\la^{(i)}$ the partition $\mu$ (if it exists) such that $\mu_j=\la_j$ for $j\neq i$ and $\mu_i=\la_i +1$. Similarly for any integer $1 \le i \le l(\la)$, we denote by $\la_{(i)}$ the partition $\nu$ (if it exists) such that $\nu_j=\la_j$ for $j\neq i$ and $\nu_i=\la_i -1$. 

In the study of the symmetric groups~\cite{La}, the following differential system is encountered. To any partition $\rho$ we associate a function $\psi_\rho(u)$ with the conditions
\begin{align*}
\sum_{i=1}^{l(\rho)+1}\frac{d}{du} \,\psi_{\rho^{(i)}}(u)&=
\sum_{i=1}^{l(\rho)+1}(\rho_i-i+1) \,\psi_{\rho^{(i)}}(u),\\
\sum_{i=1}^{l(\rho)+1}(\rho_i-i+1)\frac{d}{du} \,\psi_{\rho^{(i)}}(u)&=
\sum_{i=1}^{l(\rho)+1}(\rho_i-i+1)^2 \,\psi_{\rho^{(i)}}(u)+t|\rho|\psi_\rho(u)-(1-t)\sum_{i=1}^{l(\rho)}
\psi_{\rho_{(i)}}(u).
\end{align*}
This first order (overdetermined) differential system must be solved with the initial conditions $\psi_{\rho}(0)=0$.

In this section we shall only consider the elementary case where $\rho=(r, 1^s)$ is a hook partition. In this situation the differential system becomes
\begin{align}
\frac{d}{du} \Big(\psi_{r+1,1^s}(u)+\psi_{r,2,1^{s-1}}(u)+\psi_{r,1^{s+1}}(u)\Big)&=
r\psi_{r+1,1^s}(u)-(s+1)\psi_{r,1^{s+1}}(u),\\\notag
\frac{d}{du} \Big (r\psi_{r+1,1^s}(u)-(s+1)\psi_{r,1^{s+1}}(u)\Big)&=
r^2\psi_{r+1,1^s}(u)+(s+1)^2\psi_{r,1^{s+1}}(u)\\+t(r+s)\psi_{r,1^s}(u)&-(1-t)\Big(\psi_{r-1,1^s}(u)+\psi_{r,1^{s-1}}(u)\Big).
\end{align}

The reader may check that for $t=1$ the solutions $\psi_{r,1^s}(u)$ are given by 
\begin{align*}
(r+s)!\,\psi_{r,1^s}(u)&=(e^u-1)^{r-1}\, (e^{-u}-1)^s\\
&=\sum_{i=-s}^{r-1}(-1)^{r+s+i-1}\binom{r+s-1}{r-i-1}e^{iu}.
\end{align*}
For $t$ arbitrary, the following partial results give an idea about the high complexity of this problem.

Let us restrict to the most elementary situation $s=0$. By linear combination, the differential system (5.1)-(5.2) is easily transformed to
\begin{align}
(r+1)\,\psi^{\prime}_{r+1}&=r(r+1)\,\psi_{r+1}+
tr\psi_{r}-(1-t)\psi_{r-1},\\
(r+1)\,\psi^{\prime}_{r,1}&=-(r+1)\,\psi_{r,1}-tr\psi_{r}+(1-t)\psi_{r-1},
\end{align}
which must be solved with the initial conditions $\psi_{r}(0)=\psi_{r,1}(0)=0$.

\begin{prop}
We have
\[
(-1)^{r}r!\,\psi_{r}(u)=\sum_{i=1}^{r-1}
\genfrac{\langle}{\rangle}{0pt}{}{r-1}{r-i-1} \genfrac{\langle}{\rangle}{0pt}{}{-1}{i-1} \, e^{iu}-\genfrac{\langle}{\rangle}{0pt}{}{r-2}{r-2}.\]
\end{prop}
\begin{proof}The statement is easily checked for $r\le 3$ since we have
\[
\psi_{1}(u)=0,\quad 2\psi_{2}(u)=e^u -1,\quad-6\psi_{3}(u)=t(-e^{2u}+2e^u -1).
\]
Inspired by the $t=1$ case, we may look for a solution of (5.3) under the form 
\[(-1)^{r}r!\,\psi_{r}(u)=\sum_{i=0}^{r-1} a_i^{(r)} e^{iu}.\]
By identification of the coefficients of exponentials, we obtain
\[\frac{r-i-1}{r-1}a_i^{(r)}=ta_i^{(r-1)}+(1-t)a_i^{(r-2)}.\]
For $1\le i\le r-2$, by induction on $r$ this relation yields
 \[a_i^{(r)}=\genfrac{\langle}{\rangle}{0pt}{}{r-1}{r-i-1} \genfrac{\langle}{\rangle}{0pt}{}{-1}{i-1}.\] 
For $i=0$ by induction on $r$ we get similarly
\[a_0^{(r)}=-\genfrac{\langle}{\rangle}{0pt}{}{r-2}{r-2}.\]
For $i=r-1$ the value of $a_{r-1}^{(r)}$ is not defined. But the lattter may be obtained from the initial condition
\[\psi_{r}(0)=\sum_{i=0}^{r-1} a_i^{(r)}=0.\]
Actually applying the generalized Chu-Vandermonde formula of Corollary 3, written with $x=-1$, $y=r-1$ and $k=r-2$, we have
\begin{align*}
-a_{r-1}^{(r)}&=\sum_{i=1}^{r-2}
\genfrac{\langle}{\rangle}{0pt}{}{r-1}{r-i-1} \genfrac{\langle}{\rangle}{0pt}{}{-1}{i-1}-\genfrac{\langle}{\rangle}{0pt}{}{r-2}{r-2}\\
&=\sum_{i=0}^{r-3}\genfrac{\langle}{\rangle}{0pt}{}{r-1}{r-i-2} \genfrac{\langle}{\rangle}{0pt}{}{-1}{i}-\sum_{i=0}^{r-2}\genfrac{\langle}{\rangle}{0pt}{}{r-1}{r-i-2} \genfrac{\langle}{\rangle}{0pt}{}{-1}{i}\\
&=-\genfrac{\langle}{\rangle}{0pt}{}{-1}{r-2}.
\end{align*} 
\end{proof}

\begin{prop}
We have
\[
(-1)^{r+1}(r+1)!\,\psi_{r,1}(u)=\sum_{i=1}^{r-1}\frac{r-i}{i+1}
\genfrac{\langle}{\rangle}{0pt}{}{r}{r-i} \genfrac{\langle}{\rangle}{0pt}{}{-1}{i-1} \,e^{iu}-r\genfrac{\langle}{\rangle}{0pt}{}{r-1}{r-1}+\frac{r+1}{1-t}\genfrac{\langle}{\rangle}{0pt}{}{r}{r+1}\,e^{-u}.\]
\end{prop}
\begin{proof}
It is similar to the previous one. Inspired by the $t=1$ case, we look for a solution of (5.4) under the form 
\[(-1)^{r+1}(r+1)!\,\psi_{r,1}(u)=\sum_{i=-1}^{r-1} b_i^{(r)} e^{iu}.\]
By identification of the coefficients of exponentials, we obtain
\[(i+1)b_i^{(r)}=r(ta_i^{(r)}+(1-t)a_i^{(r-1)}).\]
For $1\le i\le r-1$ it yields
\begin{align*}
(i+1)b_i^{(r)}&=r\left(t\genfrac{\langle}{\rangle}{0pt}{}{r-1}{r-i-1} +(1-t)\genfrac{\langle}{\rangle}{0pt}{}{r-2}{r-i-2} \right)\genfrac{\langle}{\rangle}{0pt}{}{-1}{i-1}\\
&=(r-i)\genfrac{\langle}{\rangle}{0pt}{}{r}{r-i} \genfrac{\langle}{\rangle}{0pt}{}{-1}{i-1}.
\end{align*}
Similarly for $i=0$ we get
\[b_0^{(r)}=-r\left(t\genfrac{\langle}{\rangle}{0pt}{}{r-2}{r-2} +(1-t)\genfrac{\langle}{\rangle}{0pt}{}{r-3}{r-3} \right)=-r\genfrac{\langle}{\rangle}{0pt}{}{r-1}{r-1}.\]
For $i=-1$ the value of $b_{-1}^{(r)}$ is not defined. But we may obtain
\[b_{-1}^{(r)}=\frac{r+1}{1-t}\genfrac{\langle}{\rangle}{0pt}{}{r}{r+1}\] 
from the initial condition
\[\psi_{r,1}(0)=\sum_{i=-1}^{r-1} b_i^{(r)}=0.\]
Actually applying Corollary 3 and Proposition 1, we have
\begin{align*}
-b_{-1}^{(r)}&=\sum_{i=1}^{r-1}\frac{r-i}{i+1}
\genfrac{\langle}{\rangle}{0pt}{}{r}{r-i} \genfrac{\langle}{\rangle}{0pt}{}{-1}{i-1}-r\genfrac{\langle}{\rangle}{0pt}{}{r-1}{r-1}\\
&=\sum_{i=1}^{r}\frac{r(i+1)-(r+1)i}{i+1}
\genfrac{\langle}{\rangle}{0pt}{}{r}{r-i} \genfrac{\langle}{\rangle}{0pt}{}{-1}{i-1}-r\genfrac{\langle}{\rangle}{0pt}{}{r-1}{r-1}\\
&=\sum_{i=0}^{r}\frac{-(r+1)i}{i+1}
\genfrac{\langle}{\rangle}{0pt}{}{r}{r-i} \genfrac{\langle}{\rangle}{0pt}{}{-1}{i-1}=-\frac{r+1}{1-t}\genfrac{\langle}{\rangle}{0pt}{}{r}{r+1}.
\end{align*}
\end{proof}

Unfortunately the situation becomes quickly very messy and a general formula for $\psi_{r,1^{s}}(u)$ is as yet unknown. The following case is obtained by putting $s=1$ in (5.2), which leads to define $\psi_{r,1^{2}}(u)$ by
\begin{multline*}
r\psi^{\prime}_{r+1,1}(u)-2 \psi^{\prime}_{r,1^{2}}(u)=r^2\psi_{r+1,1}(u)+
4 \psi_{r,1^{2}}(u)\\
+t(r+1)\psi_{r,1}(u)-(1-t)\big(\psi_{r-1,1}(u)+\psi_{r}(u)\big),
\end{multline*}
with the initial condition $\psi_{r,1^{2}}(0)=0$. Inspired by the $t=1$ case, we look for a solution under the form 
\[(-1)^{r}(r+2)!\,\psi_{r,1^2}(u)=\sum_{i=-2}^{r-1} c_i^{(r)} e^{iu},\]
and we obtain the recurrence relation
\[-2(i+2)c_i^{(r)}=r(r-i)b_i^{(r+1)}-(r+1)(r+2)\big(tb_i^{(r)}+(1-t)(b_i^{(r-1)}+a_i^{(r)})\big).\]
For instance
\begin{align*}
-4c_0^{(r)}&=r^2b_0^{(r+1)}-(r+1)(r+2)\big(tb_0^{(r)}+(1-t)(b_0^{(r-1)}+a_0^{(r)})\big)\\
&=-r^2(r+1)\genfrac{\langle}{\rangle}{0pt}{}{r}{r}+(r+1)(r+2)\left(tr\genfrac{\langle}{\rangle}{0pt}{}{r-1}{r-1}+(1-t)((r-1)+1)\genfrac{\langle}{\rangle}{0pt}{}{r-2}{r-2}\right)\\
&=2r(r+1)\genfrac{\langle}{\rangle}{0pt}{}{r}{r}.
\end{align*}
The reader may check that the solutions are given by
\begin{align*}
&(-1)^{r}(r+2)!\, \psi_{r,1^2}(u)=\\
&\sum_{i=1}^{r-1}\left(\frac{(r-i)(r-i+1)}{(i+1)(i+2)}
\genfrac{\langle}{\rangle}{0pt}{}{r+1}{r-i+1} 
-(1-t)\frac{i(r+1)(r+2)}{2(i+1)(i+2)}
\genfrac{\langle}{\rangle}{0pt}{}{r-1}{r-i-1}\right)
\genfrac{\langle}{\rangle}{0pt}{}{-1}{i-1}e^{iu}\\
&-\binom{r+1}{2}\genfrac{\langle}{\rangle}{0pt}{}{r}{r}+
\binom{r+2}{2}\left(\frac{2e^{-u}}{1-t}
\genfrac{\langle}{\rangle}{0pt}{}{r+1}{r+2}-\genfrac{\langle}{\rangle}{0pt}{}{r-1}{r}e^{-u}+\frac{e^{-2u}}{1-t}
\genfrac{\langle}{\rangle}{0pt}{}{r}{r+2}\right).
\end{align*} 

\section*{Acknowledgements}
It is a pleasure to thank Christian Krattenthaler for a proof of Theorem 1.

\end{document}